\documentclass[10pt]{article}
\usepackage{amssymb}
\usepackage{amsmath}
\usepackage{latexsym}
\usepackage{amsthm}
\usepackage{bm}
\usepackage{amsthm}
\usepackage{chngpage}
\usepackage{calc}
\usepackage{eso-pic}
\usepackage{array,multirow,graphicx}
\usepackage{multicol}
\usepackage{amsfonts}
\usepackage{fancyhdr}
\usepackage{tikz}
\usepackage{amsopn}
\usepackage{rotating}
\usepackage[utf8]{inputenc}
\usepackage[makeroom]{cancel}
\usepackage{hyperref}

\usepackage{algorithm}
\usepackage[noend]{algpseudocode}
\usepackage{setspace}

\usepackage{caption}
\usepackage[nottoc]{tocbibind}
\usepackage{mathtools}
\usepackage{mathrsfs}
\usepackage{color} 
\usepackage[T1]{fontenc}

\usepackage{subfigure}

\theoremstyle{definition}
\newtheorem{rem}{Remark}[section]

\theoremstyle{plain}

\newtheorem{thm}{Theorem}[section]

\newtheorem{lem}[thm]{Lemma}

\renewcommand{\epsilon}{\varepsilon}

\newcommand{\floor}[1]{\left\lfloor #1 \right\rfloor}
\newcommand{\R}{\mathbb{R}}
\renewcommand{\P}{\Pi}
\newcommand{\sign}{\operatorname{sign}}

\newcommand{\C}{\mathbb{C}}

\newcommand{\dt}{\text{d}t}
\renewcommand{\phi}{\varphi}

\newcommand{\G}{\mathcal{G}}
\renewcommand{\S}{\mathcal{S}}

\numberwithin{equation}{section}

\title{Refined decay bounds on the entries of spectral projectors associated with sparse Hermitian matrices}
\author{Michele Benzi\footnote{Scuola Normale Superiore, Piazza dei Cavalieri, 7, 56126 Pisa, Italy (michele.benzi@sns.it).} ~and Michele Rinelli
	\footnote{Scuola Normale Superiore, Piazza dei Cavalieri, 7, 56126 Pisa, Italy (michele.rinelli@sns.it).}}
\date{}

\graphicspath{{plots/}}

\begin{document}
\maketitle
\begin{abstract}
	\noindent Spectral projectors of Hermitian matrices play a key role in many applications, and especially in electronic structure computations. Linear scaling methods for gapped systems are based on the fact that these special matrix functions are localized, which means that the entries decay exponentially away from the main diagonal or with respect to more general sparsity patterns. The relation with the sign function together with an integral representation is used to obtain new decay bounds, which turn out to be optimal in an asymptotic sense. The influence of isolated eigenvalues in the spectrum on the decay properties is also investigated and a superexponential behaviour is predicted.
\end{abstract}


\section{Introduction}
The a priori knowledge of decay bounds for matrix functions of banded or sparse matrices is important for many applications and has been the subject of many papers over the years. 
An exponential decay holds in general for $f(A)$, where $A$ is Hermitian and banded (or sparse) and $f$ is analytic over an ellipse containing the spectrum of $A$ \cite{BenziGolub}. Specific bounds are given for important matrix functions, like the matrix inverse \cite{Baskakov,Demko} or entire functions, like the matrix exponential, which exhibit  superexponential decay \cite{BenziSimoncini,Iserles}. Further results for classes of functions defined by an integral transform, such as Laplace-Stieltjes and Cauchy-Stieltjes functions, are given in
\cite{BenziSimoncini,Frommer2017}, where the analysis makes use of results for the inverse or the exponential. Less regular functions, like fractional powers, lead to power-law decays and are used to describe non-local dynamics; see \cite{Igor,Riascos}.

Another important case is the spectral projector of a banded Hermitian matrix, which is the orthogonal projector onto the subspace spanned by the eigenvectors associated with the eigenvalues below a certain value \cite{BBR}. This projector, also known as the density matrix in the chemistry and physics literature, is of central importance in electronic structure computations \cite{Miyazaki,Kohn,Niklasson2011}. 
In \cite{BBR} one can find rigorous proofs of exponential decay for gapped systems, like insulators. One approach is based on the approximation of the step function with the Fermi-Dirac function, the other is inspired by \cite{ChuiHasson,Hasson} and makes use of the polynomial approximation of piecewise constant function over the union of disjoint intervals.

Most of the existing bounds for $f(A)$ depend only on partial information about the spectrum of $A$, for example, the spectral interval $[\lambda_{\min}(A),\lambda_{\max}(A)]$ if $A$ is Hermitian and positive definite \cite{BenziSimoncini,Demko} or the field of values in the general case  \cite{BenziFOV,PozzaSimoncini}. For the spectral projector, a key role is played by the spectral gap, see below. However, numerical experiments show that the bounds are often pessimistic and do not capture the actual decay behaviour, which seems to depend also on the distribution of the eigenvalues within the spectral sets. A first step in this direction is taken in \cite{Frommer2018}, where the authors show a connection between the decay in the inverse of a positive definite Hermitian matrix and the distribution of the eigenvalues near the upper end of the spectrum.

In this paper we will make use of the expression of the spectral projector in terms of the matrix sign function to refine the existing bounds by exploiting an integral representation of the sign function. This will allow us to analyze how the distribution of the eigenvalues of the original Hermitian matrix affects the rate of decay in the entries of the associated spectral projector.
In particular, we will show a connection between the decay properties and the eigenvalue distribution.

The paper is organized as follows. In section \ref{section:preliminaries} we recall  basic definitions and the standard techniques that have been used to obtain decay bounds. In section \ref{section:previouswork} we recall existing decay bounds for the inverse function and spectral projectors. In section \ref{sec:newboundsproj} we give new decay bounds for spectral projectors. In section \ref{section:boundsSL} we show how the eigenvalue distribution is connected with the decay properties for  spectral projectors.


\section{Preliminaries}\label{section:preliminaries}
Let us recall some definitions and previous results concerning the localization in matrix functions of Hermitian matrix arguments. For a detailed survey, see \cite{BenziCime}.


\subsection{Matrices with exponential decay}
We say that a sequence of $n$$\times$$ n$ matrices $A_n$ has the \emph{exponential off-diagonal decay property} if there are constants $C>0$ and $\alpha>0$ independent of $n$ such that
\begin{align*}
|[A_n]_{ij}|\leq Ce^{-\alpha|i-j|},\quad \text{for all $i,j$.}
\end{align*}
Corresponding to each matrix $A_n$ we define for a nonnegative integer $m$ the matrix $A_n^{(m)}=([A_n^{(m)}]_{ij})$ as follows:
\begin{align*}
[A_n^{(m)}]_{ij}=\begin{cases}
[A_n]_{ij}\quad &\text{if $|i-j|\leq m$},\\
0\quad &\text{otherwise}.
\end{cases}
\end{align*}
Each matrix $A_n^{(m)}$ is $m$-banded since $[A_n^{(m)}]_{ij}=0$ for $|i-j|>m$. Moreover, if $A_n$ has the exponential decay property, then for all $\epsilon>0$ there is an $\bar{m}$ independent of $n$ such that $\|A_n-A_n^{(m)}\|_1\leq \epsilon$ for $m\geq \bar{m}$. See \cite{BenziRazouk} for more details. The same conclusion holds for the $\infty$-norm by considering the sequence $A_n^*$, and for the $2$-norm, owing to the inequality $\|A\|_2\leq \sqrt{\|A\|_1\|A\|_\infty}$. Obviously, being able to approximate a full matrix with exponential decay with a banded matrix (with bandwidth independent of the matrix dimension) can lead to huge computational savings.

The foregoing considerations can be extended to matrices with more general decay patterns. For instance, let $\G_n$ be a sequence of graphs with $\{1,2,\dots,n \}$ as the set of nodes and graph distances $d_n(i,j)$ \cite{Diestel}. We say that a sequence of $n$$\times$$n$ matrices $A_n$ has the \emph{exponential decay property relative to the graph} $\G_n$ if there are constants $C>0$ and $\alpha>0$ independent of $n$ such that
\begin{align*}
|[A_n]_{ij}|\leq Ce^{-\alpha d_n(i,j)}\quad \text{for all $i,j$.}
\end{align*}
In this more general setting, some restrictions on the graphs must be imposed to obtain sparse approximations of the matrix sequence, see \cite{Frommer2021} for more details.


\subsection{Connection between polynomial approximation and decay properties}
A classical approach to derive decay bounds for a matrix function $f(A)$ is to bound the error of the best uniform polynomial approximation of $f$ over a suitable set containing the spectrum of $A$. Denote with $\P_k$ the set of all polynomials with degree at most $k$. Denote the error of the best uniform approximation in $\P_k$ of a function $f$ continuous over a set $\S$ as
\begin{align}
E_k(f,\S)=\inf_{P_k\in\P_k}\sup_{z\in \S}|f(z)-P_k(z)|.\label{eqn:polynomialerror}
\end{align}
Notice that if $\S$ is a real compact interval and $f$ is real valued over $\S$ and continuous, then there exists a unique solution to the minimization problem (\ref{eqn:polynomialerror}), which becomes a minimum \cite{Meinardus}. In general (\ref{eqn:polynomialerror}) is not a minimum.

An argument that is often used \cite{BenziCime,BBR,BenziGolub,Frommer2017} in order to obtain decay bounds for matrix functions is described in the following lemma.

\begin{lem}\label{lem:boundgeneric}
	Let $A\in\mathbb{C}^{n\times n}$ be Hermitian and $m$-banded with  $\sigma(A)\subset\S$, and let $f(z)$ be defined over $\S$. Let $i,j$ be two indices such that $i\neq j$ and let $k:=\floor{\frac{|i-j|}{m}}$. Then
	\begin{align}
	|[f(A)]_{ij}|\leq E_k(f,\S).\label{eqn:boundpoly1}
	\end{align}
	If there are $C>0$ and $0<\rho<1$ such that
		\begin{align}
	E_k(f,\S)\leq C\rho^{k},\quad \text{for all }k\ge 0,\label{eqn:polyapproxgeometric}
	\end{align} 
	then
	\begin{align*}
	|[f(A)]_{ij}|\le C\rho^{\frac{|i-j|}{m}-1} \quad \text{for $i\neq j$}.
	\end{align*}
\end{lem}

\begin{proof}
	Let $P_k\in \P_k$. Then $[P_k(A)]_{ij}=0$ since $P_k(A)$ is $km$-banded and $|i-j|>km$. Therefore
	\begin{align*}
	|[f(A)]_{ij}|=|[f(A)]_{ij}-[P_k(A)]_{ij}|\leq \|f(A)-P_k(A)\|_2&=\max_{x\in \sigma(A)}|f(x)-P_k(x)|.\\
	&\leq \max_{x\in \S}|f(x)-P_k(x)|.
	\end{align*} 
	Since the inequality holds for any $P_k\in\P_k$, by the definition of $E_k(f,\S)$ we conclude that (\ref{eqn:boundpoly1}) holds.

	The second part follows from the inequality   $\floor{\frac{|i-j|}{m}}\geq \frac{|i-j|}{m}-1$ and (\ref{eqn:boundpoly1}).
\end{proof}

A bound like (\ref{eqn:polyapproxgeometric}) for the uniform polynomial approximation holds if $\S$ is a closed interval and $f$ can be extended to an analytic function over an ellipse strictly containing $\S$ \cite{Meinardus}, but it can hold also for more general domains \cite{ChuiHasson,Meinardus}.

\begin{rem}
	The same approach works for a general matrix $A$ by taking $k=d(i,j)-1$, where $d(i,j)$ is the geodesic distance between $i$ and $j$ in the graph associated with the matrix $A$. See \cite{BBR,BenziRazouk,Frommer2021}. In fact we have that $[P_k(A)]_{ij}=0$ for any $P_k\in\P_k$ and $i,j$ such that $d(i,j)>k$, then the proof proceeds as in Lemma \ref{lem:boundgeneric}. Although all the results of the next sections will be given only for banded case, they also hold for general sparsity patterns by slightly modifying the estimates.
\end{rem}

\begin{rem}
	The result of Lemma \ref{lem:boundgeneric} implies that if we are given a sequence of $n$$\times$$n$ matrices $\{A_n\}$ of increasing size, all Hermitian, uniformly $m$-banded and such that $\sigma(A_n)\subset \S$ for all $n$, then $f(A_n)$ is well defined for all $n$ and the bound (\ref{eqn:boundpoly1}) holds for all the matrices in the sequence, since it depends only on the set $\S$ and on the bandwidth, and not on $n$.
\end{rem}


\section{Previous work}\label{section:previouswork}
Here we recall some known decay bounds for the matrix inverse and for spectral projectors.


\subsection{Decay bounds for the inverse}
The error for the polynomial approximation of the inverse function over an interval is explicitly known \cite{Meinardus}.

\begin{thm}\label{thm:Demko_poly}
	Consider the function $1/x$ defined over $[a,b]$, where $0<a<b$. Let
	\begin{align}
	r=\frac{b}{a},\quad C=\frac{(1+\sqrt{r})^2}{2b},\quad q = \frac{\sqrt{r}-1}{\sqrt{r}+1}.\label{eqn:Demkoparameters}
	\end{align}
	Then
	\begin{align}
		E_k(1/x,[a,b])=Cq^{k+1}\quad \text{for all }k. \label{eqn:inversepolyerror}
	\end{align}
\end{thm}

Theorem \ref{thm:Demko_poly} has been used in \cite{Demko} to obtain the following result regarding the entries of the matrix inverse.

\begin{thm}
	Let $A\in\mathbb{C}^{n\times n}$ be Hermitian, positive definite and $m$-banded. Let $a=\lambda_{\min}(A)$, $b= \lambda_{\max}(A)$. 
		Let $r,C,q$ be defined as in (\ref{eqn:Demkoparameters}). Then, for any $i$ and $j$ such that $i\neq j$, 
	\begin{align}
	|[A^{-1}]_{ij}|\leq Cq^{\frac{|i-j|}{m}}.\label{eqn:Demko_inverse}
	\end{align}
\end{thm}

In \cite{Demko} the authors choose a different constant in order to capture the case $i=j$. In fact $|[A^{-1}]_{ii}|\leq \|A^{-1}\|_2=1/a$ for any $i$, so if we choose the maximum between $1/a$ and the value of $C$ in (\ref{eqn:Demkoparameters}) we obtain a bound which holds for any $i,j$. Here it is more convenient to distinguish the two cases.

A reader familiar with Krylov methods will recognize in the expression for $q$ given in (\ref{eqn:Demkoparameters}) the geometric rate of the bound for the error reduction (measured in the $A$-norm) of the conjugate gradient method applied to a linear system $Ax=b$ with a positive definite $A$. See, for example, \cite{Strakos}. 
It is also well known that 
this bound can be overly pessimistic, and that much faster convergence can occur for certain distributions of the eigenvalues of $A$, for instance when the eigenvalues are clustered in the lower end of the spectrum. The following result \cite{Frommer2018} shows that this phenomenon holds also for the decay in the entries of the inverse.

\begin{thm}\label{thm:Frommersuperlinear}
	Let $A\in\mathbb{C}^{n\times n}$ be Hermitian, positive definite and $m$-banded with eigenvalues $\lambda_1\le \lambda_2\le \dots\le\lambda_n$. Let
	\begin{align*}
	r_{\ell}=\frac{\lambda_{n-\ell}}{\lambda_1},\quad q_{\ell}=\frac{\sqrt{r_{\ell}}-1}{\sqrt{r_\ell}+1},\quad C=\frac{2}{\lambda_1}.
	\end{align*}
	Then the entries of $A^{-1}$ are bounded by
	\begin{align}
	|[A^{-1}]_{ij}|\leq Cq_\ell^{\frac{|i-j|}{m}-\ell} \quad \text{for all }\ell = 0,1,\dots,\floor{\frac{|i-j|}{m}}.\label{eqn:FrommerDecaySL}
	\end{align}
\end{thm}

The family of bounds given in Theorem \ref{thm:Frommersuperlinear} tells us that we can remove some eigenvalues in the upper end of the spectrum and obtain a bound like the one in (\ref{eqn:Demko_inverse}) with a smaller geometric rate, but paying the price of a smaller exponent. This means that, if some of the largest eigenvalues are isolated, the decay can be predicted much more accurately than (\ref{eqn:Demko_inverse}), which is a special case of (\ref{eqn:FrommerDecaySL}) with $\ell=0$ up to a constant factor. We will return on this in Section \ref{section:boundsSL}.


\subsection{Properties of spectral projectors}

The ability to approximate spectral projectors associated with banded or sparse matrices is crucial to the development of linear scaling methods in electronic structure computations; see \cite{BBR,Miyazaki,Kohn,Niklasson2011}.

Let $H\in \C^{n\times n}$ be Hermitian with eigenvalues $\lambda_1\leq  \dots\leq \lambda_{n_e}<\lambda_{n_e+1}\leq \dots\leq \lambda_{n}$, and let $\mathbf{v}_i$, $i = 1, \ldots, n_e$ be an orthonormal basis for the $H$-invariant subspace associated with the first $n_e$ eigenvalues (counting multiplicities). Then the spectral projector associated with this $H$-invariant subspace can be represented as
	\begin{align*}
	P=\mathbf{v}_1\mathbf{v}_1^*+\dots+\mathbf{v}_{n_e}\mathbf{v}_{n_e}^*=\sum_{i=1}^{n_e}\mathbf{v}_i\mathbf{v}_i^*.
	\end{align*}

In electronic structure computations, we deal with sequences of matrices of increasing size $\{ H_n \}$ and their projectors $\{P_n\}$. In particular, each $H_n$ is Hermitian. These matrices arise as a Galerkin discretization of a continuous operator, and  $n = n_b\,\cdot\, n_e$, where $n_b$ is the number of basis functions used for the projection and is fixed, while $n_e$ is the number of electrons of the starting system and increases. See \cite{BBR} for more details.

In order to derive a common exponential decay for all the
projectors $P_n$ we need that
\begin{itemize}
	\item the matrices $H_n$ have uniformly bounded bandwidth;
	\item there exist four parameters $b_1<a_1<a_2<b_2$ independent of $n$ such that $\sigma(H_n)\subset[b_1,a_1]\cup[a_2,b_2]$ and, for all $n_e$, $[b_1,a_1]$ contains the first $n_e$ eigenvalues of $H_n$ while $[a_2,b_2]$ contains the remaining $n-n_e$.
\end{itemize}
The key quantity here is the relative spectral gap $\gamma = (a_2-a_1)/(b_2-b_1)$. If this quantity is not too small (e.g., insulators or semiconductors) then the projector exhibits exponential decay, while for small or vanishing gap (e.g., metallic systems) the decay cannot be fast \cite{BBR}.

Under the assumptions above, any projector can be written as $P_n=h(H_n)$, where $h(x)$ is the Heaviside function:
\begin{align}
h(x)=\begin{cases}
1\quad &\text{if $x<\mu$},\\
\tfrac{1}{2}\quad &\text{if $x=\mu$},\\
0\quad &\text{if $x>\mu$},
\end{cases}\label{eqn:heaviside}
\end{align}
where $\mu$ is such that $a_1<\mu<a_2$. 
Throughout this paper, we will consider a single Hermitian, banded matrix $H\in \C^{n\times n}$ with spectrum contained in $[b_1,a_1]\cup[a_2,b_2]$ and its projector $P=h(H)$, but keeping in mind that in the applications it is an element of a matrix sequence satisfying the two properties above.

Since the Heaviside function is discontinuous over the interval $[b_1,b_2]$, the quantity $E_k(h,[b_1,b_2])$ does not converge to $0$, so we cannot use directly Lemma \ref{lem:boundgeneric} to obtain decay bounds for $P$. As a first approach, in \cite{BBR} 
it is considered the approximation of $h(x)$ over the spectrum of $H$ with the Fermi-Dirac function
\begin{align*}
f_{FD}(x)=\frac{1}{1+e^{\beta(x-\mu)}},
\end{align*}
that is analytic over a family of ellipses containing $[b_1,b_2]$, so the entries of $f_{FD}(H)$ decay exponentially \cite{BBR,BenziGolub}. Usually, a large value of $\beta$ is required to have  a good approximation, and this leads to pessimistic decay bounds.

Since the spectral projector of $H$ depends only on the eigenvectors associated with the first $n_e$ eigenvalues, a scaled and shifted modification such as
\begin{align*}
\tilde{H}=c\,H+d\,I,\quad c>0,\quad d\in\R,
\end{align*}
has the same  spectral projector.  This allows us to make convenient assumptions on the spectrum. For instance, if  $d=-(a_2+a_1)/2$ then $\sigma(\hat{H})\subset[\hat{b}_1,-a]\cup[a,\hat{b}_2]$, 
with $a=(a_2-a_1)/2$, so we can choose $\mu=0$ in (\ref{eqn:heaviside}). Then, by putting $b=\max \{\hat{b}_2,-\hat{b}_1\}$, we have $\sigma(\hat{H})\subset[-b,-a]\cup[a,b]$. When dealing with a sequence of matrices, the transformation must be the same for all the matrices. This can be done under the assumptions above.

Another approach, which turns out to be better than the one based on the Fermi-Dirac approximation, consists in estimating directly the error of the polynomial approximation of the Heaviside function over $[-b,-a]\cup[a,b]$. If $\mu=0$, then the identity $h(x)=(1-\sign(x))/2$ holds, where
\begin{align*}
\sign(x)=\begin{cases}
-1\quad &\text{if $x<0$},\\
0\quad &\text{if $x=0$},\\
1\quad &\text{if $x>0$}.
\end{cases}
\end{align*}
Moreover, $\sign(x)=x/|x|=x/(x^2)^{\frac{1}{2}}$ for any $x\neq 0$. The main idea is to consider a polynomial $Q_k(x)\in\P_{k}$ which approximates $x^{-\frac{1}{2}}$ over $[a^2,b^2]$, and then construct $P_{2k+1}(x)=\frac{1}{2}(1-xQ_k(x^2))$ to approximate $h(x)$. This leads to the following result \cite{BBR}.

\begin{thm}\label{thm:boundprojold}
	Let $H$ be Hermitian and $m$-banded with $\sigma(H)\subset [-b,-a]\cup[a,b]$, and let $P=h(H)$ be the spectral projector associated with the negative eigenvalues of $H$. Then, for $1<\xi<\bar{\xi}:=\frac{b+a}{b-a}$, we have 
	\begin{align}
	|[P]_{ij}| \leq \frac{2b\xi M(\xi)}{\xi-1}\left( \frac{1}{\xi} \right)^{\frac{|i-j|}{2m}} \quad \text{for all $i,j$,}\label{eqn:boundprojold}
	\end{align}
	where 
	\begin{align*}
		M(\xi)=\frac{1}{\sqrt{z_0}},\quad z_0=\left[ \frac{b^2+a^2}{b^2-a^2}-\frac{\xi^2+1}{2\xi} \right]\frac{b^2-a^2}{2}.
	\end{align*}
\end{thm}

\begin{rem}
	 Theorem \ref{thm:boundprojold} gives us a parametrized family of bounds and not a single one like in (\ref{eqn:Demko_inverse}) for the inverse. For values of $\xi$ near to $\bar{\xi}$, the constant factor in (\ref{eqn:boundprojold}) blows up and the bound becomes unusable. What can be done is to optimize for any $i,j$ the right-hand side in (\ref{eqn:boundprojold}) among all the possible values of $\xi$, so one obtains
	\begin{align}
	|[P]_{ij}|\leq \inf_{1<\xi<\bar{\xi}}\left[\frac{2b\xi M(\xi)}{\xi-1}\left( \frac{1}{\xi} \right)^{\frac{|i-j|}{2m}}\right].\label{eqn:boundprojoldopt}
	\end{align}
\end{rem}

\begin{rem}
Other features of the spectral projector follow from the fact that $P=P^2$. An important consequence of this identity is that $|[P]_{ij}|\leq 1$ for any $i,j$. This means that any bound for the entries of $P$ is useless unless it is less than $1$ too.  
\end{rem}


\section{New decay bounds for spectral projectors}\label{sec:newboundsproj}

In this section we establish new decay bounds for the spectral projector $P=h(H)$, where $H$ is banded, Hermitian and with spectrum contained in the union of two symmetric intervals and $h(x)$ is the Heaviside function defined as in (\ref{eqn:heaviside}) with $\mu=0$.
For this purpose, since the identity $h(x)=\frac{1}{2}(1-\sign(x))$ holds,  it is equivalent to study the decay properties of $\sign(H)$ instead of $P$. Numerical validation of the results is given at the end of this section with some experiments.


\subsection{Exploiting an integral representation of the sign function}\label{sectionsignintegral}
Let $H\in\mathbb{C}^{n\times n}$ be Hermitian and banded with $\sigma(H)\subset [-b,-a]\cup[a,b]$, where $0<a<b$.
Consider the representation for $\sign(x)$ \cite{Higham}
\begin{align*}
\sign(x)=\frac{2}{\pi}\int_0^\infty \frac{x}{x^2+t^2}\,\dt,
\end{align*}
which leads to 
\begin{align}
\sign(H)=\frac{2}{\pi}\int_0^\infty H(H^2+t^2I)^{-1}\,\dt.\label{eqn:signintegral}
\end{align}
The integral (\ref{eqn:signintegral}) is well defined componentwise since 
$$
|[H(H^2+t^2I)^{-1}]_{ij}|\leq \|H(H^2+t^2I)^{-1}\|_2\leq \frac{b}{a^2+t^2}
$$
and the right-hand side is integrable.
From (\ref{eqn:signintegral}), one can bound the entries of $\sign(H)$ as follows:
\begin{align}
|[\sign(H)]_{ij}|\leq \int_0^\infty | [ H(H^2+t^2 I)^{-1} ]_{ij}| \,\dt.\label{eqn:sign_integral_ineq}
\end{align}
Let $f_t(x)=x(x^2+t^2)^{-1}$, so we have $f_t(H)=H(H^2+t^2I)^{-1}$. Notice that $f_t(x)=xg_t(x^2)$ where $g_t(x)=(x+t^2)^{-1}$.
In order to bound the entries of $f_t(H)$ we can use the following results.

\begin{lem}\label{lem:oddpolyapprox}
	Let $f(x)=xg(x^2)$ be defined for $x\in [-b,-a]\cup[a,b]$, where $g(x)$ is continuous over $[a^2,b^2]$. Suppose that
	\begin{align*}
	E_k(g, [a^2,b^2])\le Cq^{k},\quad \text{for all }k\ge 0,	\end{align*}
	where $C>0$, $0<q<1$ are independent of $k$. Then, for any $k\geq 1$,
	\begin{align}
	&E_k(f,[-b,-a]\cup[a,b])\le b\cdot Cq^{\frac{k-1}{2}}\quad \text{for $k$ odd,}\label{eqn:Ek_with_k_odd}\\
	&E_k(f,[-b,-a]\cup[a,b])\le b\cdot Cq^{\frac{k-2}{2}}\quad \text{for $k$ even}.\nonumber
	\end{align}
	In particular:
	\begin{align}
	E_k(f,[-b,-a]\cup[a,b])\le b\cdot Cq^{\left \lfloor{\frac{k-1}{2}}\right \rfloor }\leq b\cdot Cq^{\frac{k-2}{2}} \quad \text{for all $k\ge0$}.\label{eqn:Ek_odd_and_even}
	\end{align}
\end{lem}

\begin{proof}
	Suppose that $k$ is odd, so $k=2s+1$ with $s\geq0$.
	Since $g$ is continuous, there exists $Q_s\in \P_s$  such that 
	\begin{align*}
	E_s(g,[a^2,b^2])=\max_{x\in [a^2,b^2]}|g(x)-Q_s(x)|.
	\end{align*}
	Let $P_{k}(x):=xQ_s(x^2)$. Since $P_{k}\in\P_{k}$ we have
	\begin{align*}
	E_{k}(f,[-b,-a]\cup[a,b])&\leq \max_{x\in [-b,-a]\cup[a,b]}|f(x)-P_k(x)|\\
	&=\max_{x\in [-b,-a]\cup[a,b]}|x\left(g(x^2)-Q_s(x^2)\right)|\\
	&\leq b\max_{x\in [a^2,b^2]}|g(x)-Q_s(x)| \\
	&=b E_{s}(g,[a^2,b^2])\\
	&\leq bCq^{s}=bCq^{\frac{k-1}{2}}.
	\end{align*}
	If $k$ is even, then $k-1$ is odd and  we can use the inequality
	\begin{align*}
	E_k(f,[-b,-a]\cup[a,b])\leq E_{k-1}(f,[-b,-a]\cup[a,b])\leq bCq^{\frac{k-2}{2}}.
	\end{align*} 
	From these two inequalities and $\left\lfloor \frac{k-1}{2} \right\rfloor\geq \frac{k-2}{2}$ we obtain (\ref{eqn:Ek_odd_and_even}).
\end{proof}

Now we can combine Lemma \ref{lem:oddpolyapprox} and Lemma \ref{lem:boundgeneric} in order to obtain bounds for the entries of $H(t^2I+H^2)^{-1}$.

\begin{lem}\label{lem:decayintegrandsign}
	Let $H\in\mathbb{C}^{n\times n}$ be Hermitian and $m$-banded with $\sigma(H)\subset[-b,-a]\cup[a,b]$. For any $t\geq 0$, let
		\begin{align}
 	r(t)=\frac{b^2+t^2}{a^2+t^2},\quad
 	C(t)=\frac{(1+\sqrt{r(t)})^2}{2(b^2+t^2)},\quad 	q(t)=\frac{\sqrt{r(t)}-1}{\sqrt{r(t)}+1}.\label{eqn:Demkoparameters2}
	\end{align}
	Then
	\begin{align}
	|[H(H^2+t^2)^{-1}]_{ij}|\leq bC(t)q(t)^{\frac{|i-j|}{2m}-\frac{1}{2}}\quad \text{for } i\neq j. \label{eqn:decayintegrandsign}
	\end{align}
\end{lem}

\begin{proof}
	Consider the identity $f_t(x)=xg_t(x^2)$ with $g_t(x)=(x+t^2)^{-1}$. Since $E_k((x+t^2)^{-1},[a^2,b^2])=E_k(x^{-1},[a^2+t^2,b^2+t^2])$, from (\ref{eqn:inversepolyerror}) we have
	\begin{align*}
	E_k(g_t,[a^2,b^2])=C(t)q(t)^{k+1}=C(t)q(t)\cdot q(t)^{k},
	\end{align*}
	for all $k$. Therefore,
	by applying Lemma \ref{lem:oddpolyapprox},
	\begin{align*}
	E_k(f_t,[-b,-a]\cup[a,b])\leq bC(t)q(t)q(t)^{\frac{k-2}{2}}=bC(t)q(t)^{\frac{k}{2}}.
	\end{align*}
	Since $i\neq j$ and $H$ is $m$-banded, from Lemma \ref{lem:boundgeneric}  we obtain that
	\begin{align*}
	|[f_t(H)]_{ij}|\leq bC(t)q(t)^{\frac{|i-j|}{2m}-\frac{1}{2}}.
	\end{align*}
	This concludes the proof.
\end{proof}

Now we can bound the entries of $\sign(H)$.

\begin{thm}\label{thmsign1}
	Let $H$ be Hermitian and $m$-banded with $\sigma(H)\subset[-b,-a]\cup[a,b]$. Let $C(t)$, $q(t)$ be defined as in (\ref{eqn:Demkoparameters2}). Then, for $|i-j|\geq m$,
	\begin{align}
	|[\sign(H)]_{ij}|\leq \frac{2b}{\pi}\int_0^{\infty}C(t)q(t)^{\frac{|i-j|}{2m}-\frac{1}{2}}\,\dt,\label{boundsigngeneric}
	\end{align}
	and
	\begin{align}
	|[\sign(H)]_{ij}|\leq \hat{C}\hat{q}^{\frac{|i-j|}{2m}-\frac{1}{2}},\label{boundsigngeometric}
	\end{align}
	where
	\begin{align*}
		\hat{C}=\frac{1}{2}\left( 1+\sqrt{\frac{b}{a}} \right)^2,\quad
	\hat{q}=q(0)= \frac{b-a}{b+a}.
	\end{align*}
\end{thm}

\begin{proof}
	The inequality (\ref{boundsigngeneric}) follows directly from (\ref{eqn:sign_integral_ineq}) and Lemma \ref{lem:decayintegrandsign}. For (\ref{boundsigngeometric}), it holds that  
	\begin{align*}
	\frac{2b}{\pi}
	\int_0^{\infty}C(t)q(t)^{\frac{|i-j|}{2m}-\frac{1}{2}}\,\dt
	&\leq\left( \frac{2b}{\pi}\int_0^{\infty}C(t)\,\dt \right) q(0)^{\frac{|i-j|}{2m}-\frac{1}{2}} .
	\end{align*}
	In order to estimate the integral, we can expand $C(t)$ in the integral as follows:
	\begin{align*}
	&\int_0^\infty C(t)\,\dt=\int_0^\infty \frac{(1+\sqrt{r(t)})^2}{2(b^2+t^2)}\,\dt\\&=\frac{1}{2}\left( \int_0^{\infty}\frac{1}{b^2+t^2}\,\dt+\int_0^{\infty}\frac{1}{a^2+t^2}\,\dt+2\int_0^{\infty}\frac{1}{\sqrt{(b^2+t^2)(a^2+t^2)}}\,\dt \right),
	\end{align*}
	where $r(t)$ is given in (\ref{eqn:Demkoparameters2}).
	The first two integrals can be explicitly computed:
	\begin{align*}
	&\int_{0}^{\infty}\frac{1}{b^2+t^2}\,\dt=\frac{\pi}{2b},\quad \int_0^{\infty}\frac{1}{a^2+t^2}\,\dt=\frac{\pi}{2a}.
	\end{align*}
	The third can be bounded by using the Cauchy-Schwarz inequality:
	\begin{align*}
	\int_0^{\infty}\frac{1}{\sqrt{(b^2+t^2)(a^2+t^2)}}\,\dt&\leq \left( \int_{0}^{\infty}\frac{1}{b^2+t^2}\,\dt  \right)^{\frac{1}{2}}\cdot \left( \int_{0}^{\infty}\frac{1}{a^2+t^2}\,\dt  \right)^{\frac{1}{2}}\\
	&=\frac{\pi}{2\sqrt{ab}}.
	\end{align*}
	Therefore
	\begin{align}
	\frac{2b}{\pi}\int_0^{\infty}C(t)\,\dt\le \frac{1}{2}\left( 1+\frac{b}{a}+2\sqrt{\frac{b}{a}} \right)=\frac{1}{2}\left(1+\sqrt{\frac{b}{a}}\right)^{2}.\label{intineq}
	\end{align}
	This concludes the proof.
\end{proof}

By exploiting the relation between the spectral projector and the matrix sign, we obtain the following result.

\begin{thm}\label{thm:boundproj1}
	Let $H\in\mathbb{C}^{n\times n}$ be as in Theorem \ref{thmsign1}, and let $P=\frac{1}{2}(I-\sign(H))$ be the spectral projector associated with the negative eigenvalues of $H$. Then
	\begin{align}
		|[P]_{ij}|\leq \hat{C}\hat{q}^{\frac{|i-j|}{2m}-\frac{1}{2}},\label{eqn:boundprojnew1}
	\end{align}
	where 
	\begin{align*}
	\hat{C}=\frac{1}{4}\left(1+\sqrt{\frac{b}{a}}\right)^2,\quad \hat{q}=\frac{b-a}{b+a}. 
	\end{align*}
\end{thm}

\begin{proof}
	If $|i-j|\geq m$, the inequality follows directly from Theorem \ref{thmsign1} and the identity $|[P]_{ij}|=|[\sign(H)]_{ij}|/2$. For $|i-j|<m$, note that $|[P]_{ij}|\leq 1$ for all $i,j$ \cite{BBR} and the right hand side of (\ref{eqn:boundprojnew1}) is greater than $1$.
\end{proof}

\begin{rem}
	This result improves the bound in (\ref{thm:boundprojold}) since all the geometric rates are smaller than the one in (\ref{eqn:boundprojnew1}).  Moreover, one doesn't need to choose the better estimate among a family of bounds.
\end{rem}


\subsection{An asymptotically optimal bound}
Although the bound given in Theorem (\ref{thmsign1}) behaves well in practice, it is not optimal from an asymptotic point of view.
Hasson showed in \cite{Hasson} that there exists $C>0$ such that 
\begin{align}
E_k(\sign(x),[-b,-a]\cup [a,b])\leq \frac{C}{\sqrt{k}}\left( \frac{b-a}{b+a} \right)^\frac{k}{2}.\label{eqn:Hasson_rate}
\end{align}
By Lemma \ref{lem:boundgeneric}, this leads to 
\begin{align}
|[\sign(H)]_{ij}|\leq \frac{C}{\sqrt{\frac{|i-j|}{m}-1}}\left( \frac{b-a}{b+a} \right)^{\frac{|i-j|}{2m}-\frac{1}{2}},\label{eqn:signoptimal}
\end{align}
that is asymptotically faster than the bound in (\ref{eqn:boundprojnew1}). This is actually the  best result we can obtain by using polynomial approximations of the sign function, since
\begin{align*}
\sqrt{k}\left( \frac{b+a}{b-a} \right)^{\frac{k}{2}}\cdot E_k(\sign(x),[-b,-a]\cup[a,b])=\mathcal{O}(1)\quad \text{as $k\to\infty$}.
\end{align*}
See \cite{Eremenko2007} for more details.
Here the disadvantage is that is not possible to compute or estimate the constant $C$.
We will obtain, by manipulating the integral in (\ref{boundsigngeneric}), a decay that is asymptotically equivalent to (\ref{eqn:signoptimal}) but with computable parameters.

In order to obtain better bounds we start with the inequality (\ref{boundsigngeneric}). In the proof of Theorem \ref{thmsign1}, a key argument is the inequality
\begin{align}
q(t)^{\frac{|i-j|}{2m}-\frac{1}{2}}\leq q(0)^{\frac{|i-j|}{2m}-\frac{1}{2}}\quad \text{for any $t\geq 0$}.\label{q(t)ineq1}
\end{align}
The next result gives us a better estimate of the left-hand side in (\ref{q(t)ineq1}).

\begin{lem}\label{estimategeometricintegrand}
	Let $q(t)$ be defined as in (\ref{eqn:Demkoparameters2}) and let $\alpha>0$ be real. Then 
	\begin{align*}
	q(t)^\alpha \leq  e^{-\alpha t^2(C_1-t^2C_2)}q(0)^\alpha \quad \text{for all } t\geq0,
	\end{align*}
	where
	\begin{align*}
	C_1=\frac{1}{2ab},\quad C_2 = \frac{a^2+ab+b^2}{8a^3b^3}.
	\end{align*}
	Moreover, for any fixed $\tau$ such that $0<\tau<\sqrt{\frac{C_1}{C_2}}$, we have
	\begin{align}
	q(t)^\alpha\leq e^{-\alpha t^2(C_1-\tau^2 C_2)}q(0)^\alpha  \quad \text{for any } 0\leq t\leq \tau.\label{estimateintegrand}
	\end{align}
\end{lem}

\begin{proof}
	Let $t\geq 0$ be fixed. We have
	\begin{align*}
	q(t)^\alpha=\left( \frac{\sqrt{b^2+t^2}-\sqrt{a^2+t^2}}{\sqrt{b^2+t^2}+\sqrt{a^2+t^2}} \right)^\alpha\leq \left(\frac{1}{b+a}\right)^\alpha\cdot \left(\sqrt{b^2+t^2}-\sqrt{a^2+t^2}\right)^\alpha,
	\end{align*}
	so
	\begin{align}
	\frac{q(t)^\alpha}{q(0)^\alpha}\leq   \frac{\left( \sqrt{b^2+t^2}-\sqrt{a^2+t^2}\right)^\alpha }{\left(b-a\right)^\alpha}.\label{qt_ineq_Gaussian}
	\end{align}
	Consider the function $s(x)=\sqrt{b^2+x}-\sqrt{a^2+x}$, and its Taylor expansion with Lagrange remainder centered in $0$:
	\begin{align*}
	s(x)=b-a -\frac{1}{2}\frac{b-a}{ab}x+\frac{1}{6}s''(\xi)x^2,\quad 0\leq \xi\leq x.
	\end{align*}
	Since 
	\begin{align*}
	s''(\xi)=\frac{3}{4}\left( \frac{1}{(a^2+\xi)^{3/2}}-\frac{1}{(b^2+\xi)^{3/2}} \right)&\leq \frac{3}{4}\left( \frac{1}{a^3}-\frac{1}{b^3} \right)\\
	&=\frac{3}{4}\frac{(b-a)(a^2+ab+b^2)}{a^3b^3}, 
	\end{align*}we have
	\begin{align*}
	s(x)&\leq (b-a)\left[ 1-\frac{1}{2ab}x+\frac{1}{8}\frac{a^2+ab+b^2}{a^3b^3}x^2 \right]\\
	&=(b-a)(1-C_1x+C_2x^2).
	\end{align*}
	Since the numerator in (\ref{qt_ineq_Gaussian}) is $s(t^2)^\alpha$, we have
	\begin{align*}
	\frac{q(t)^\alpha}{q(0)^\alpha}&\leq ( 1-C_1t^2+C_2t^4 )^\alpha\\
	&=e^{\alpha\log(1-C_1t^2+C_2t^4)}\\
	&\leq e^{\alpha(-C_1t^2+C_2t^4)}=e^{-\alpha t^2(C_1-C_2t^2)},
	\end{align*}
	where for the last inequality we have used that $\log(x)\leq x-1$ for all $x>0$.
	
	For (\ref{estimateintegrand}), observe that if $0\leq t\leq \tau$ then $C_1-t^2C_2$ achieves its minimum in $\tau$, so $C_1-t^2C_2\geq C_1-\tau^2C_2$ for $0\leq t\leq \tau$. This concludes the proof.
\end{proof}

\begin{rem}
	The inequality (\ref{estimateintegrand}) is uniform in $t$ as long as $0\le t\le \tau$. Moreover, for $\tau<\sqrt{\frac{C_1}{C_2}}$, we have that $C_1-\tau^2C_2>0$. This means that $q(t)^\alpha $ is bounded by $q(0)^\alpha$ times a Gaussian factor.
\end{rem}

Now we can proceed with the estimate for the entries of $\sign(H)$.

\begin{thm}\label{thm:boundsignoptimal}
	Let $H\in\mathbb{C}^{n\times n}$ be Hermitian and $m$-banded with $\sigma(H)\subset [-b,-a]\cup[a,b]$. Let $C_1=\frac{1}{2ab}, C_2 = \frac{a^2+ab+b^2}{8a^3b^3}$ and $0<\tau<\bar{\tau}:=\sqrt{\frac{C_1}{C_2}}$. Then
	\begin{align*}
	|[\sign(H)]_{ij}|\leq \frac{K_1(\tau)}{\sqrt{\frac{|i-j|}{m}-1}}q(0)^{\frac{|i-j|}{2m}-\frac{1}{2}}+K_2q(\tau)^{\frac{|i-j|}{2m}-\frac{1}{2}},
	\end{align*}
	where $q(t)$ is defined as in (\ref{eqn:Demkoparameters2}) and
	\begin{align*}
	&K_1(\tau) =\sqrt{\frac{2}{\pi}}\left(1+\frac{b}{a}\right)^2\cdot\frac{1}{\sqrt{C_1-\tau^2 C_2}},\\
	&K_2=\frac{1}{2}\left(1+\sqrt{\frac{b}{a}}\right)^{\frac{1}{2}}.
	\end{align*}
\end{thm}

\begin{proof}
	Let $\alpha :=\frac{|i-j|}{2m}-\frac{1}{2}$. From the hypothesis $|i-j|\geq m$ we have $\alpha\geq 0$.
	From Theorem \ref{thmsign1}, we have 
	\begin{align}
	\begin{aligned}
	|[\sign(H)]_{ij}|&\leq \frac{2b}{\pi}\int_0^{\infty}C(t)q(t)^{\alpha}\,\dt.
	\end{aligned}\label{eqn:recallsignbound}
	\end{align}
	 Then we split the integral in two terms as follows:
	\begin{align*}
	\int_0^{\infty}C(t)q(t)^\alpha\,\dt&=\int_0^{\tau}C(t)q(t)^\alpha\,\dt+\int_{\tau}^{\infty}C(t)q(t)^\alpha\,\dt.
	\end{align*}
	The first term can be bounded by using Lemma \ref{estimategeometricintegrand}, as $t$ ranges from $0$ to $\tau$, and the inequality $C(t)\leq C(0)$:
	\begin{align*}
	\frac{2b}{\pi}\int_0^{\tau}C(t)q(t)^\alpha\,\dt&\leq  \frac{2b}{\pi}C(0)\int_0^{\tau}q(t)^{\alpha}\,\dt\\
	&\leq \frac{2b}{\pi}C(0)q(0)^{\alpha}\int_0^{\tau}e^{-\alpha (C_1-\tau^2C_2)t^2}\,\dt\\
	&\leq \frac{2b}{\pi}C(0)q(0)^{\alpha}\int_0^{\infty}e^{-\alpha (C_1-\tau^2C_2)t^2}\,\dt\\
	&=\frac{b}{\pi}C(0)q(0)^{\alpha}\cdot \sqrt{ \frac{\pi}{\alpha(C_1-\tau^2C_2)} }\\
	&=\frac{1}{\sqrt{\pi}}\left(1+\frac{b}{a}\right)^2\frac{1}{\sqrt{C_1-\tau^2 C_2}}\frac{1}{\sqrt{\frac{|i-j|}{2m}-\frac{1}{2}}}	q(0)^{\alpha},
	\end{align*}
	where we have used that $\int_0^{\infty}e^{-\sigma x^2}\,dx=\frac{\sqrt{\pi}}{2\sqrt{\sigma}}$ for any $\sigma>0$.

	For the second term:
	\begin{align*}
	\frac{2b}{\pi}\int_{\tau}^{\infty}C(t)q(t)^{\alpha}\,\dt&\leq \frac{2b}{\pi}\int_0^{\infty}C(t)\,\dt\cdot q(\tau)^{\alpha}\\
	&\leq \frac{1}{2}\left(1+\sqrt{\frac{b}{a}}\right)^{2}q(\tau)^{\alpha},
	\end{align*}
	where we have used (\ref{intineq}).
	Combining these two inequalities with (\ref{eqn:recallsignbound}), we conclude.	
\end{proof}

\begin{thm}\label{thm:boundproj2}
	Let $H\in\C^{n\times n}$ be as in Theorem \ref{thm:boundsignoptimal} and let $P=(I-\sign(H))/2$ be the spectral projector associated with the negative eigenvalues of $H$. Then
	\begin{align}
	|P_{ij}|\leq \frac{1}{2}\left( \frac{K_1(\tau)}{\sqrt{\frac{|i-j|}{m}-1}}q(0)^{\frac{|i-j|}{2m}-\frac{1}{2}}+K_2q(\tau)^{\frac{|i-j|}{2m}-\frac{1}{2}} \right), \label{eqn:boundprojnew2}
	\end{align}
	where all the parameters are defined as in Theorem \ref{thm:boundsignoptimal}.
\end{thm}

\begin{rem}\label{rem:optimizebound2}
	Since $q(\tau)<q(0)$ for any $\tau>0$, the second term in  (\ref{eqn:boundprojnew2})  decays faster than the first. Hence, the asymptotic behaviour of this bound is equal to the one in  (\ref{eqn:signoptimal}), but with computable parameters. This new bound depends on the choice of $\tau$, that ranges between $0$ and $\bar{\tau}$. As in \ref{thm:boundprojold}, we have a whole family of bounds which can be optimized among the admissible values of $\tau$. 
\end{rem}


\subsection{Comparison of existing bounds}\label{sec:experiments1}
For the next experiments we will assume that $\sigma(H)\subset[-1,-a]\cup[a,1]$, so $b=1$. This is not restrictive, since we can scale and shift the matrix in order to satisfy the condition.

Since any bound for a generic entry $[P]_{ij}$ of the spectral projector depends only on the value $|i-j|$, in order to study the exact decay of  $P$ we can consider the quantities 
\begin{align*}
D_P(k):= \max_{|i-j|=k}|[P]_{ij}|,\quad \text{for any }k\geq 0.
\end{align*}
Let us denote the bounds for $D_P(k)$ induced by (\ref{eqn:boundprojoldopt}), (\ref{eqn:boundprojnew1}) and (\ref{eqn:boundprojnew2}) for $|i-j|=k$ as $B_1(k)$, $B_2(k)$ and $B_3(k)$, respectively. The third is optimized among the admissible values of $\tau$ as described in Remark \ref{rem:optimizebound2}. The components of $P$ satisfy $|P_{ij}|\leq 1$ for all $i,j$, so it is convenient to use the following bound:
\begin{align*}
D_P(k)\leq \min\{1,B_s(k)\},
\end{align*}
for any $s=1,2,3$.

For the tests we have constructed Hermitian matrices with prescribed size, bandwidth and spectrum in the following way:
\begin{itemize}
	\item A unitary matrix $Q\in \C^{n\times n}$ is taken as the Q factor of the QR factorization of a random matrix with prescribed size.
	\item A symmetric, dense matrix is computed as $Q\Lambda Q^*$, where $\Lambda$ is diagonal with the prescribed eigenvalues.
	\item On $Q\Lambda Q^*$ we have used similarity transformations with Householder matrices as in \cite[Section 7.4.3]{Golub} in order to obtain a matrix with prescribed bandwidth. 
\end{itemize}

With this technique, we have constructed a $2000\times2000$ Hermitian matrix $H$ which is $20$-banded and such that $\sigma(H)\subset[-1,-0.3]\cup[0.3,1]$. In Figure \ref{fig:comparisonlinear}
the exact decay is compared with the bounds.
\begin{figure}[h]
	\centering
	\includegraphics[width=7cm]{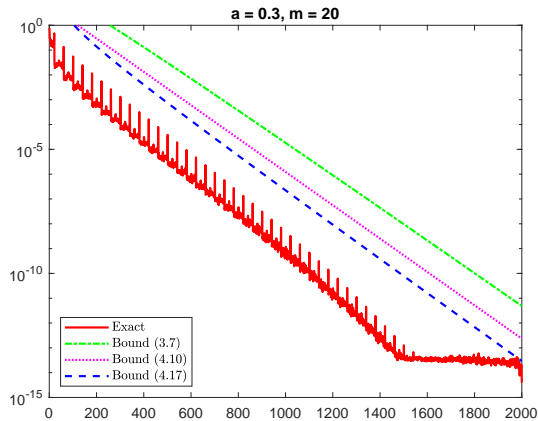}
	\caption{Logarithmic plot of the bounds compared with the exact decay for the spectral projector associated with the negative eigenvalues of a $20$-banded, $2000$$\times$$2000$ Hermitian matrix with uniformly distributed eigenvalues in $[-1,-0.3]\cup[0.3,1]$.  \label{fig:comparisonlinear}}
\end{figure}
We can see that the decay rate seems to be captured by all the bounds, although $B_3(k)<B_2(k)<B_1(k)$ for all $k$.

The bounds can be used to truncate the projector to a banded matrix with a small error. Let us see how the bounds behave in order to truncate $P$ to $P^{(m)}$ where $m$ is chosen in order to have $[P^{(m)}]_{ij}=0$ for $|P_{ij}|\leq \epsilon$ for a fixed threshold $\epsilon$.
 Let us define 
\begin{align*}
 	&m_i(\epsilon)=\min\{ \bar{k}:B_i(k)\leq \epsilon \text{ for $k\geq \bar{k}$}\},\quad i=1,2,3,	\\
 	 &m_P(\epsilon)=\min\{ \bar{k}:D_P(k)\leq \epsilon \text{ for $k\geq \bar{k}$}\}.
\end{align*}
 The value $m_i(\epsilon)$ is the first for which the bound $B_i$ becomes definitively smaller than a threshold $\epsilon$, and $m_P(\epsilon)$ does the same with $D_p$. The values of $m_i(\epsilon)$, $i=1,2,3$, and $m_P(\epsilon)$ associated with the previous example are displayed in Table \ref{tab:approxband}.
We note that $m_3(\epsilon)$ provides the best estimate in all cases. Once again, it should be emphasized that for a given accuracy $\epsilon$, the estimated bandwidth is independent of $n$.  
We can also notice that the exact decay has an oscillatory behaviour with period equal to the original bandwidth $m$. This is reflected in the fact that $m_P(\epsilon)$ is always a multiple of $m$. 
\begin{table}[h]
	\centering
	\begin{tabular}{|c| c| c |c| c| c|}
		\hline 
		&$\epsilon = 1e-1$&$\epsilon = 1e-2$&$\epsilon = 1e-3$&$\epsilon = 1e-4$&$\epsilon = 1e-5$\\
		\hline 
		$m_1(\epsilon)$&419 & 577&733 &887  &1041  
		\\
		\hline 
		$m_2(\epsilon)$&270 &419 &568 &717  & 865 
		\\
		\hline 
		$m_3(\epsilon)$&218 &347 &483 &623  &764\\
		\hline  
		$m_P(\epsilon)$&60 &180 &300 &420 & 540
		\\
		\hline 
	\end{tabular}
	\caption{\label{tab:approxband}}
\end{table}


\subsection{Other approaches}
Our approach strongly relies on the fact that $\sigma(H)$ is contained in the union of two symmetric intervals. Although this hypothesis is always satisfied by choosing suitable values of $a$ and $b$, the bound does not behave like the real decay when $b$ (or $-b$) is not close to the maximum (resp., minimum) eigenvalue. If $\sigma(H)\subset[-b_1,-a]\cup[a,b_2]$ with $b_1\neq b_2$, it would be preferable to consider  this domain instead of $[-b,-a]\cup[a,b]$ with $b=\max\{ b_1,b_2 \}$.

It is shown \cite{Eremenko2011,Fuchs} that there exist positive constants $C_1, C_2$ and $\eta$ such that
\begin{align}
C_1k^{-\frac{1}{2}}e^{-\eta k}\leq E_k(\sign(x),[-b_1,-a]\cup[a,b_2])\leq C_2k^{-\frac{1}{2}}e^{-\eta k}.\label{eqn:rate_Fuchs}
\end{align}
The rate $\eta$ is computable and given by
\begin{align*}
\eta =\int_{-1}^K\frac{K-x}{\sqrt{(1-x^2)(x+b_1/a)(x-b_2/a)}}\,\text{d}x,
\end{align*}
where
\begin{align*}
K=\frac{\int_{-1}^1x((1-x^2)(x+b_1/a)(x-b_2/a))^{-1/2}\,\text{d}x}{\int_{-1}^1((1-x^2)(x+b_1/a)(x-b_2/a))^{-1/2}\,\text{d}x}.
\end{align*}
However, the arguments in \cite{Fuchs} do not give the values of $C_1$ and $C_2$, so they are unknown.

As an example, we constructed a $300\times 300$, $20$-banded, Hermitian matrix $H$ such that $\sigma(H)\subset[-0.5,-0.1]\cup[0.1,1]$. In Figure \ref{fig:rateFuchs} the real decay is compared with the asymptotic rate in (\ref{eqn:rate_Fuchs}) and the rate in (\ref{eqn:signoptimal}), that is asymptotically equivalent to the bound (\ref{eqn:boundprojnew2}). All the constant factors are put to $1$  to compare only the asymptotic behaviour.
\begin{figure}[h]
	\centering
	\includegraphics[width=7cm]{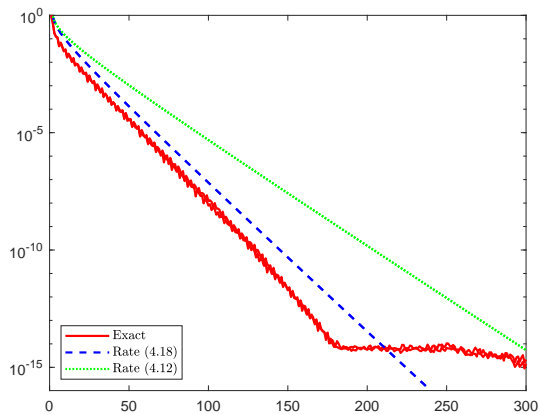}
	\caption{Logarithmic plot of the decay rates (\ref{eqn:rate_Fuchs}) and (\ref{eqn:signoptimal}) compared with the exact decay of the spectral projector associated with the negative eigenvalues of a $300\times 300$ tridiagonal matrix with spectrum in $[-0.5,-0.1]\cup[0.1,1]$.  \label{fig:rateFuchs}}
\end{figure}


\section{Bounds that take into account the eigenvalue distribution}\label{section:boundsSL}
In this section we will give bounds for the entries of spectral projectors which take account of more spectral information than the previous techniques.

First, we study the case of the matrix inverse. The result of Theorem \ref{thm:Frommersuperlinear} can be refined by directly working on the best polynomial approximation. 
\begin{thm}\label{thminverseSL}
	Let $A\in\mathbb{C}^{n\times n}$ be Hermitian positive definite and $m$-banded with distinct eigenvalues $\lambda_1< \lambda_2< \dots < \lambda_\nu$, with $\nu\leq n$. For $\ell<\nu$ define $\sigma_\ell(H)=\{ \lambda_1,\dots,\lambda_{\nu-\ell} \}$ and 
	\begin{align*}
	r_\ell=\frac{\lambda_{\nu-\ell}}{\lambda_1},\quad q_{\ell}=\frac{\sqrt{r_{\ell}}-1}{\sqrt{r_{\ell}}+1},\quad C_{\ell}=\frac{\left( 1+\sqrt{r_\ell} \right)^2}{2\lambda_{\nu-\ell}}.
	\end{align*}
	Then
	\begin{align}
	E_k(1/x,\sigma(A))\leq E_{k-\ell}(1/x,\sigma_\ell(H))\leq C_{\ell}q_\ell^{k+1-\ell},\label{inverseapproxSL}
	\end{align}
	for all $k$ and $\ell=0,\dots,k$. Moreover, we have that
	\begin{align}
	|[A^{-1}]_{ij}|\leq C_{\ell}q_{\ell}^{\frac{|i-j|}{m}-\ell},\label{inverseboundSL}
	\end{align}
	for $|i-j|>m$ and $\ell = 0,1,\dots,\floor{\frac{|i-j|}{m}}$.
\end{thm}

\begin{proof}
Let $Q_{k-\ell}$ be a polynomial of degree $k-\ell$. Define
	\begin{align*}
	R_{\ell}(x)=\prod_{i=\nu-\ell+1}^\nu\left(1-\frac{x}{\lambda_i}\right)
	\end{align*}
	and let
	\begin{align}
	P_k(x)=\frac{1}{x}(1-R_\ell(x))+R_\ell(x)Q_{k-\ell}(x).\label{compositepoly}
	\end{align}
	Since $R_\ell(0)=1$, we have that $1-R_\ell(x)$ is a multiple of $x$ so the first term in (\ref{compositepoly}) is a polynomial of degree $\ell-1$. Then $P_k(x)$ has degree $k$. From the identity
	\begin{align*}
	\frac{1}{x}-P_k(x)=R_\ell(x)\left( \frac{1}{x}-Q_{k-\ell}(x) \right)
	\end{align*}
	and by using that $R_\ell(\lambda_i)=0$ for $i=n-\ell+1,\dots,n$ and $|R_\ell(x)|\leq 1$ for any $x\in\sigma_\ell(H)$, we have
	\begin{align*}
	E_k(1/x,\sigma(A))\leq \max_{x\in \sigma(A)}\left| \frac{1}{x}-P_k(x) \right|&=\max_{x\in \sigma(A)}\left[ |R_\ell(x)|\cdot \left| \frac{1}{x}-Q_{k-\ell}(x) \right| \right]\\
	&\leq \max_{x\in \sigma_\ell(A)}\left| \frac{1}{x}-Q_{k-\ell}(x) \right|.
	\end{align*}
	Since the inequality holds for any $Q_{k-\ell}(x)\in\Pi_{k-\ell}$ and by using Theorem \ref{thm:Demko_poly}, we have that
	\begin{align*}
	E_k(1/x,\sigma(H))\leq E_k(1/x,\sigma_\ell(A))&\leq E_k(1/x,[a,\lambda_{\nu-\ell}])\\ 
	&=C_{\ell}q_{\ell}^{k+1-\ell},
	\end{align*}
	so (\ref{inverseapproxSL}) holds. For (\ref{inverseboundSL}) it is sufficient to apply Lemma \ref{lem:boundgeneric}.
\end{proof}

\begin{rem}
	With Theorem \ref{thminverseSL} we have refined the result of Theorem \ref{thm:Frommersuperlinear}, since $C_\ell\leq 2/\lambda_1$ for any $\ell$. This is not a big improvement since $C_\ell$ increases with $\ell$ and the values of these two constants do not differ much in practice. Moreover, in Theorem \ref{thm:Frommersuperlinear} no attention is given to the case of multiple eigenvalues.
\end{rem}


\subsection{Sign function and spectral projector}
As in Section \ref{sectionsignintegral}, we study the matrix sign by using the integral representation (\ref{eqn:signintegral}), so we can reduce the problem to analyzing the entries of $H(H^2+t^2I)^{-1}$.

We want to extend the results valid for the inverse function concerning the decay with respect to the effective condition number to the case of spectral projectors. As in section \ref{sec:newboundsproj}, we consider a banded, Hermitian matrix with spectrum contained in two symmetric intervals and study its matrix sign. In the spirit of Theorem \ref{thm:Frommersuperlinear}, we can bound the entries of $H(H^2+t^2I)^{-1}$ using the following result.

\begin{lem}\label{lem:odd_estimate_SL}
	Let $H\in\mathbb{C}^{n\times n}$ be Hermitian and $m$-banded with $\sigma(H)\subset[-b,-a]\cup [a,b]$. Let $a=\mu_1<\mu_2<\dots<\mu_\nu=b$, with $\nu\le n$, be the distinct values of $|\lambda|$ for $\lambda\in\sigma(H)$, and let $b_\ell = \mu_{\nu-\ell}$. For any $t\geq0$ let
	\begin{align*}
	r_{\ell}(t)=\frac{b_\ell^2+t^2}{a^2+t^2},\quad 	C_{\ell}(t)=\frac{( 1+\sqrt{r_\ell(t)} )^2}{2(b_{\ell}^2+t^2)},\quad q_\ell(t)=\frac{\sqrt{r_\ell(t)}-1}{\sqrt{r_\ell(t)}+1}.
	\end{align*}
	Then
	\begin{align}
	|[H(H^2+t^2I)^{-1}]_{ij}|\leq b_\ell C_\ell(t)q_{\ell}(t)^{\frac{|i-j|}{2m}-\frac{1}{2}-\ell},\label{compositeestimateSL}
	\end{align}
	for any $\ell=0,1,\dots,\floor{\frac{|i-j|}{2m}-\frac{1}{2}}$.
\end{lem}

\begin{proof}
	From the definition of $\mu_i$, we have $\sigma(H^2+t^2I)=\{ \mu_1^2+t^2,\dots,\mu_{\nu}^2+t^2 \}$. Moreover, in the same notation of Theorem \ref{thminverseSL}, we have that $\sigma_\ell(H)=\{ \mu_1^2+t^2,\dots,\mu_{\nu-\ell}^2+t^2 \}$.
	Consider the function $1/x$ defined over  $\sigma(H^2+t^2I)$. By proceeding as in Theorem \ref{thminverseSL}, we can construct $P_k(x)\in\P_k$ such that 
	\begin{align}
	\frac{1}{x}-P_k(x)=R_\ell(x)\left( \frac{1}{x}-Q_{k-\ell}(x) \right),\label{Rlfactorize}
	\end{align} 
	where $R_\ell(x)\in \P_\ell$ satisfies $|R_\ell(\mu_i^2+t^2)|<1$ for $i=1,\dots,n-\ell$ and $R_\ell(\mu_i^2+t^2)=0$ for $i=\nu-\ell+1,\dots,\nu$, and $Q_{k-\ell}(x)\in\P_{k-\ell}$ is the polynomial of best uniform approximation for $1/x$ over the interval $[\mu_1^2+t^2,\mu_{\nu-\ell}^2+t^2]$. Then
	\begin{align*}
	\max_{x\in \sigma(H^2+t^2I)}\left| \frac{1}{x}-P_k(x)\right|
	&\leq
	\max_{x\in\sigma_\ell(H^2+t^2I)}\left| \frac{1}{x}-Q_{k-\ell}(x)\right|\\
	&\leq C_\ell(t)q_\ell(t)^{k+1-\ell}.
	\end{align*}
	In order to approximate $f_t(x)$  over $\sigma(H)$, consider $S_{2k+1}(x):=xP_k(x^2+t^2)\in\P_{2k+1}$. In view of (\ref{Rlfactorize}), we have
	\begin{align*}
	f_t(x)-S_{2k+1}(x)&=x\left( \frac{1}{x^2+t^2}-P_{k}(x^2+t^2) \right)\\
	&=xR_{\ell}(x^2+t^2)\left( \frac{1}{x^2+t^2}-Q_{k-\ell}(x^2+t^2) \right).
	\end{align*} 
	Therefore,
	\begin{align*}
	E_{2k+1}(f_t,\sigma(H))&\leq 
	\max_{x\in\sigma(H)}|f_t(x)-S_{2k+1}(x)|\\
	&=\max_{x\in\sigma(H)}\left(|x|\cdot|R_\ell(x^2+t^2)|\cdot \left| \frac{1}{x^2+t^2}-Q_{k-\ell}(x^2+t^2) \right|\right)\\
	&\leq \max_{x\in\{ \mu_1,\dots,\mu_{\nu-\ell} \}}\left(|x|\cdot \left| \frac{1}{x^2+t^2}-Q_{k-\ell}(x^2+t^2) \right|\right) \\
	&\leq b_\ell\cdot\max_{x\in\{ \mu_1,\dots,\mu_{\nu-\ell} \}} \left| \frac{1}{x^2+t^2}-Q_{k-\ell}(x^2+t^2) \right|\\
	&\leq b_\ell \,C_\ell(t)\,q_{\ell}(t)^{k+1-\ell}.
	\end{align*}
	By proceeding as in Lemma \ref{lem:oddpolyapprox}, we obtain that 
	\begin{align*}
	E_k(f_t(x),\sigma(H))\leq b_\ell \,C_\ell(t)\,q_\ell(t)^{\frac{k}{2}-\ell}.
	\end{align*}
	Then, by applying Lemma \ref{lem:boundgeneric} we conclude that (\ref{compositeestimateSL}) holds.
\end{proof}

Now we can state the analogue of Theorem \ref{thmsign1}.

\begin{thm}\label{thmsignSL}
	Under the same hypotheses of Lemma \ref{lem:odd_estimate_SL}, we have that
	\begin{align*}
	|[\sign(H)]_{ij}|\leq \hat{C}_{\ell}\hat{q}_\ell^{\frac{|i-j|}{2m}-\frac{1}{2}-\ell},
	\end{align*}
	where
	\begin{align*}
	\hat{q}_{\ell}=q_{\ell}(0)=\frac{b_\ell-a}{b_\ell+a},\quad \hat{C}_{\ell}=\frac{1}{2}\left( 1+\sqrt{\frac{b_\ell}{a}} \right)^2.
	\end{align*}
\end{thm}

\begin{proof}
	By applying
	(\ref{eqn:sign_integral_ineq}) and Lemma \ref{lem:odd_estimate_SL}, we have
	\begin{align*}
	|[\sign(H)]_{ij}|&\leq \frac{2}{\pi}\int_0^\infty C_\ell(t)q_\ell(t)^{\frac{|i-j|}{2m}-\frac{1}{2}-\ell}\,\dt\\
	&\leq \left(  \frac{2}{\pi}\int_0^\infty b_\ell C_\ell(t)\,\dt \right)q_\ell(0)^{\frac{|i-j|}{2m}-\frac{1}{2}-\ell}.
	\end{align*}
	By proceeding as in Theorem \ref{thmsign1}, we obtain the desired bound.
\end{proof}

\begin{thm}\label{thm:BoundprojSL}
	Let $H\in\mathbb{C}^{n\times n}$ be as in Theorem \ref{thmsignSL}, and let $P=\frac{1}{2}(I-\sign(H))$ be the spectral projector. Then
	\begin{align}
	|[P_{ij}]|\leq \hat{C}_{\ell }\hat{q}_{\ell}^{\frac{|i-j|}{2m}-\frac{1}{2}-\ell}.\label{eqn:projboundSL}
	\end{align}
	where
	\begin{align*}
	\hat{C}_{\ell}=\frac{1}{4}\left( 1+\sqrt{\frac{b_{\ell}}{a}} \right)^2,\quad \hat{q_{\ell}}= \frac{b_\ell-a}{b_\ell+a}.
	\end{align*}
\end{thm}

Theorem \ref{thm:BoundprojSL} gives us a family of bounds parametrized by $\ell$. Hence, for fixed $i,j$, the corresponding entry of the projector is bounded by
\begin{align}
|[P_{ij}]|\leq \min_{\ell=0,\dots, \left\lfloor \frac{|i-j|}{2m}-\frac{1}{2}\right\rfloor} \hat{C}_{\ell }\hat{q}_{\ell}^{\frac{|i-j|}{2m}-\frac{1}{2}-\ell}.\label{eqn:optimizedbound}
\end{align}
Depending on the eigenvalue distribution of $H$, this can predict a much faster decay than the results of Section \ref{sec:newboundsproj}. For instance, increasing $\ell$ gives a smaller geometric rate $\hat{q}_\ell$ but also a smaller exponent. If some of the eigenvalues that are largest in magnitude are isolated, $q_\ell$ becomes much smaller even for moderate values of $\ell$. In case of a cluster of eigenvalues near the spectral gap, we can also predict a superexponential decay. We will see some examples in the next section.

In all the results in this section, a special attention is given to the case where the absolute value $|\lambda|$ of an eigenvalue $\lambda\in\sigma(H)$ appears more than once. In practice, it is usual to have isolated eigenvalues with largest absolute value that have an high multiplicity. See \cite{BBR}.

\subsection{Numerical experiments}
Here we see how the bound (\ref{eqn:projboundSL}) works on some  examples. The matrices are generated with the method described in Section \ref{sec:experiments1}.

For the first example we consider a $2000\times 2000$, $20$-banded matrix $H$ for which $-1$ is an eigenvalue with multiplicity $10$ and all the other eigenvalues are uniformly distributed over $[-0.5,-0.1]\cup[0.1,0.5]$.
In order to apply the results of Section \ref{sec:newboundsproj} we must consider the inclusion $\sigma(H)\subset[-1,-0.1]\cup[0.1,1]$. However, if we apply Theorem \ref{thm:BoundprojSL} with $\ell=1$ we obtain $b_1 = 0.5$ that leads to a much faster bound, as we can see in Figure \ref{fig:oneisolated}.

\begin{figure}[h]
	\centering
	\includegraphics[width=7cm]{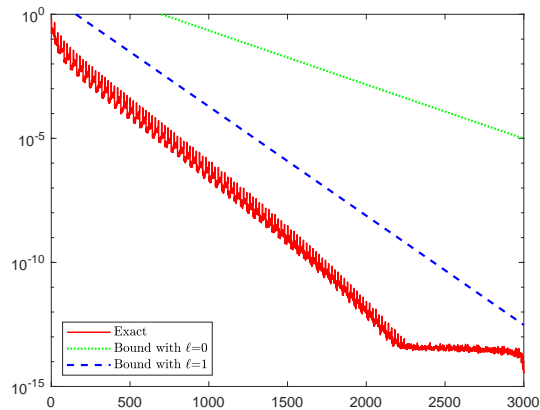}
	\caption{ Logarithmic plot of the bounds given by Theorem \ref{thm:BoundprojSL} applied with $\ell=0,1$ compared with the exact decay of the spectral projector associated with the negative eigenvalues of a $20$-banded, $3000\times 3000$ Hermitian matrix with spectrum contained in $\{-1\}\cup[-0.5,-0.1]\cup[0.1,0.5]$.  \label{fig:oneisolated}}
\end{figure}

Now we show that Theorem \ref{thm:BoundprojSL} can predict a superexponential decay behaviour if the eigenvalues are clustered near the spectral gap. We first consider the case where the spectrum is symmetric with respect to the origin, so, in the notation of Lemma \ref{lem:odd_estimate_SL}, any $\mu_i$ corresponds to two eigenvalues, one positive and one negative. More precisely, we consider a $300\times300$, tridiagonal matrix $H$ with eigenvalues 
\begin{align*}
\lambda_i^{(j)}=(-1)^j\left[ 1+0.9 \left( 1-\frac{i-1}{299}-2\sqrt{1-\frac{i-1}{299}} \right) \right]\in [-1,-0.1]\cup[0.1,1],
\end{align*}
for $i=1,\dots,150$ and $j=0,1$. In the notation of Lemma \ref{lem:odd_estimate_SL} we have that $\nu=150$ and $\mu_i=\lambda_i^{(0)}=|\lambda_i^{(1)}|$ for $i=1,\dots,150$. In Figure \ref{fig:superexponential_symmetric} the decay of the spectral projector is compared with the bounds given by Theorem \ref{thm:BoundprojSL} for $\ell=0,\dots,50$, and with a bound that is optimized among the values of $\ell$. We see that the behaviour is captured and that the optimized bound differs from the exact decay by a few orders of magnitude.


\begin{figure}[h]
	\centering
	\begin{minipage}[c]{0.46\linewidth}
\includegraphics[width=\linewidth]{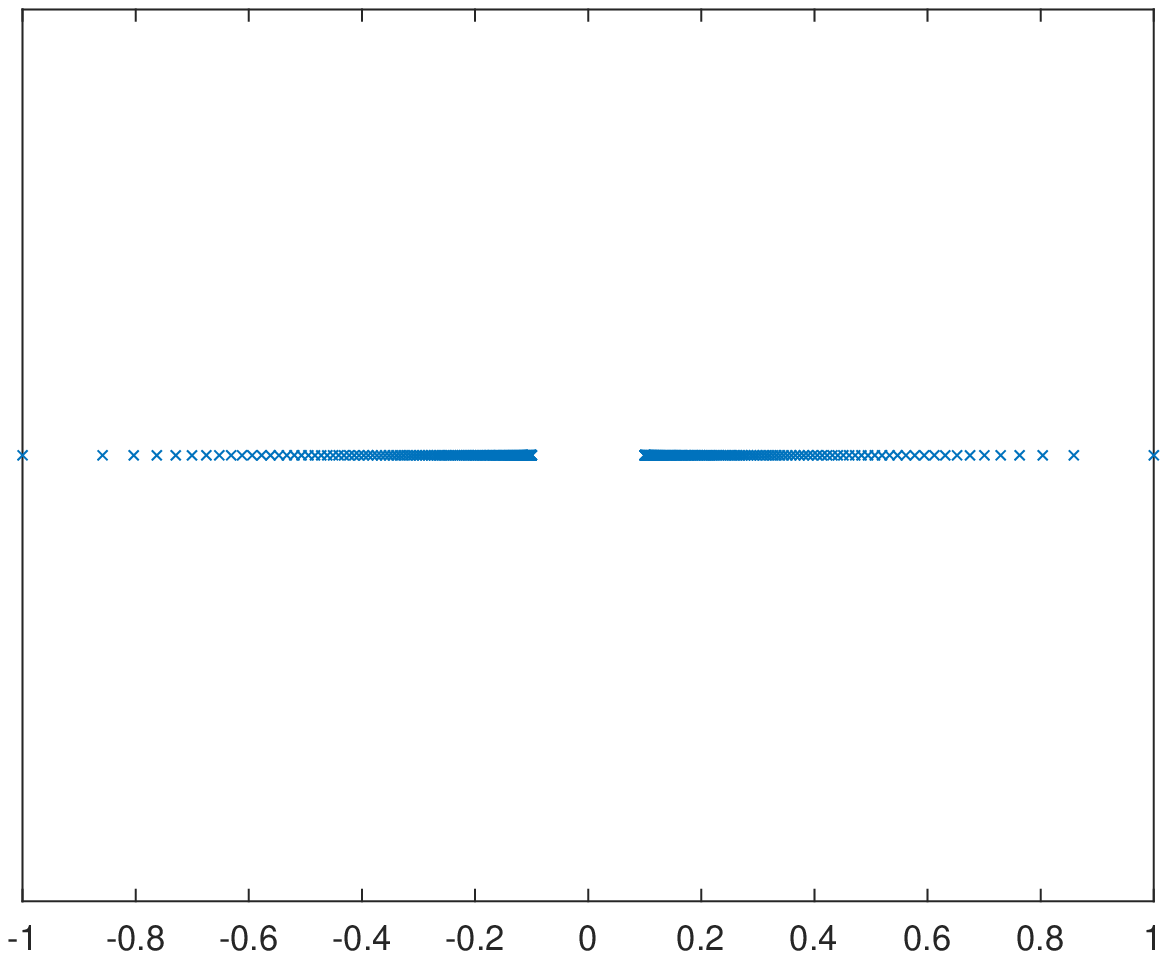}
\end{minipage}
\hfill 
\begin{minipage}[c]{0.4968\linewidth}
	\includegraphics[width=\linewidth]{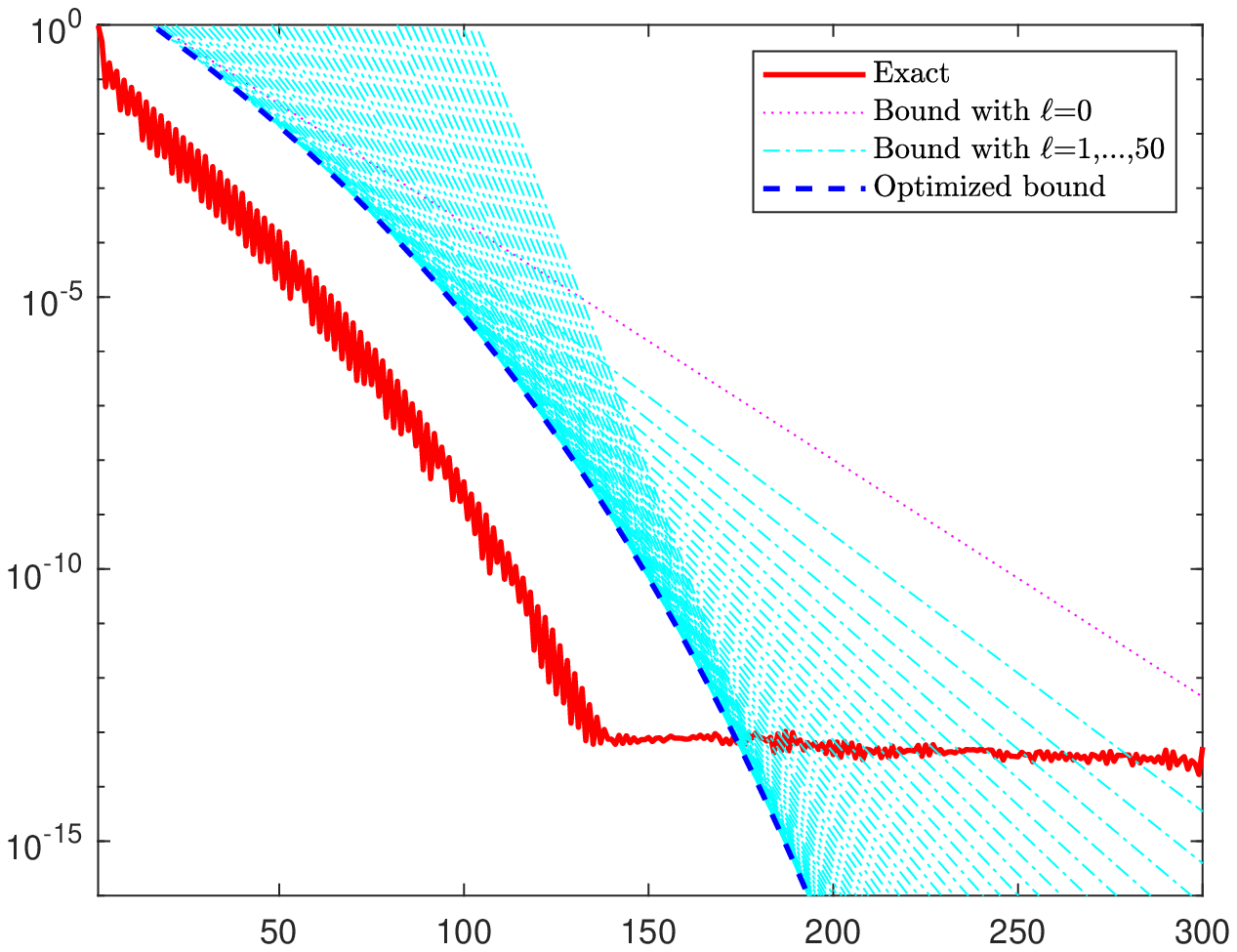}
\end{minipage}
	\caption{Left: Plot of the spectrum of $H$. Note that it is symmetric with respect to the origin and that the cluster near the spectral gap.  Right: Exact decay of the projector compared with the bounds (\ref{eqn:projboundSL}) for $\ell=0,1,\dots,50$. The dotted line, which corresponds to $\ell=0$, is the bound (\ref{eqn:boundprojnew1}). The dashed line corresponds to the best bound among the values of $\ell$. \label{fig:superexponential_symmetric}}
\end{figure}

The situation is different if the eigenvalues are not symmetric. For instance, consider a $300\times 300$, tridiagonal, Hermitian matrix $H$ with eigenvalues
\begin{align*}
\lambda_i=(-1)^i\left[ 1+0.9 \left( 1-\frac{i-1}{299}-2\sqrt{1-\frac{i-1}{299}} \right) \right]\in [-1,-0.1]\cup[0.1,1],
\end{align*}
for $i=1,\dots,300$.

For this case, the comparison is shown in Figure \ref{fig:superexponential}. We can see that the optimized bound has a superexponential decay but does not capture the exact behaviour.

\begin{figure}[h]
	\centering
\begin{minipage}[c]{0.46\linewidth}
	\includegraphics[width=\linewidth]{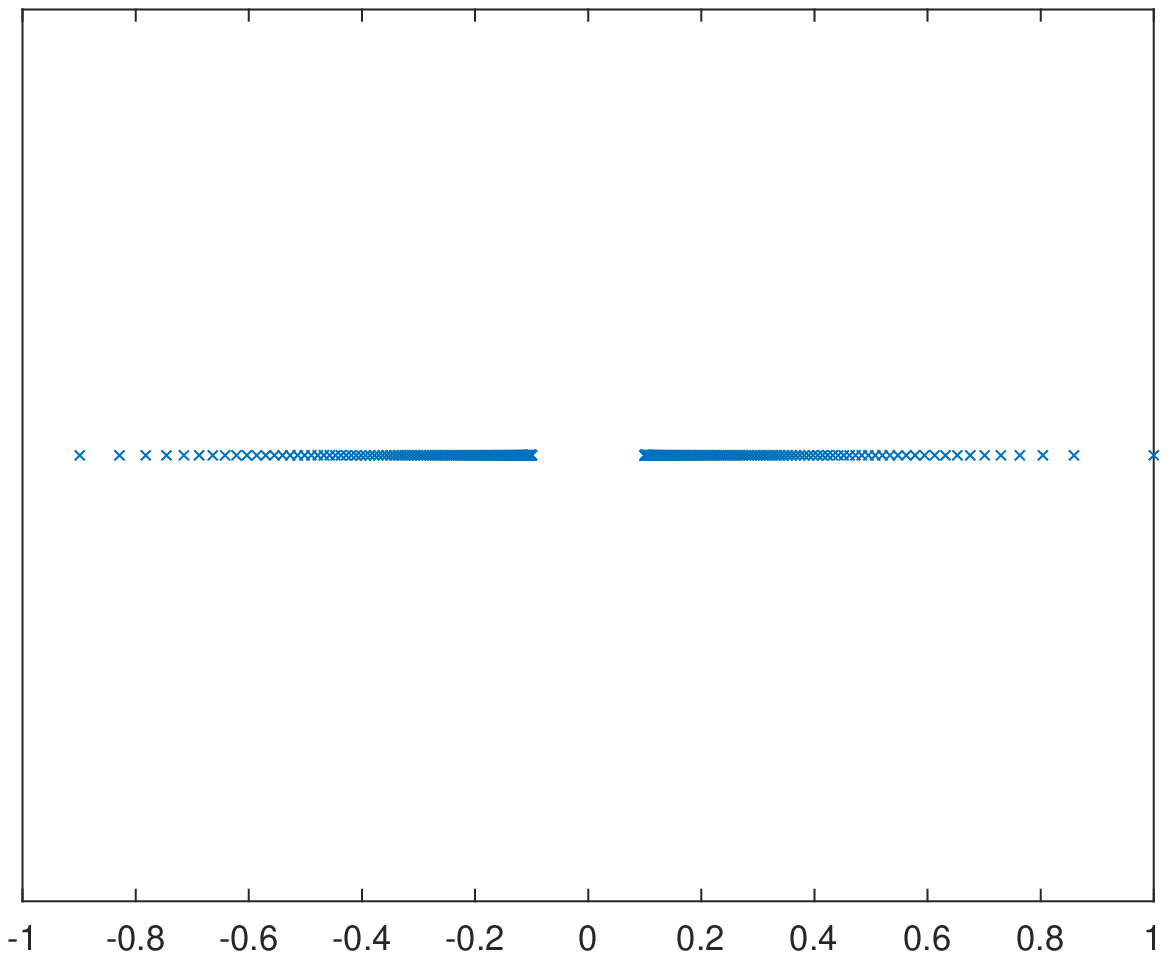}
\end{minipage}
\hfill 
\begin{minipage}[c]{0.4968\linewidth}
	\includegraphics[width=\linewidth]{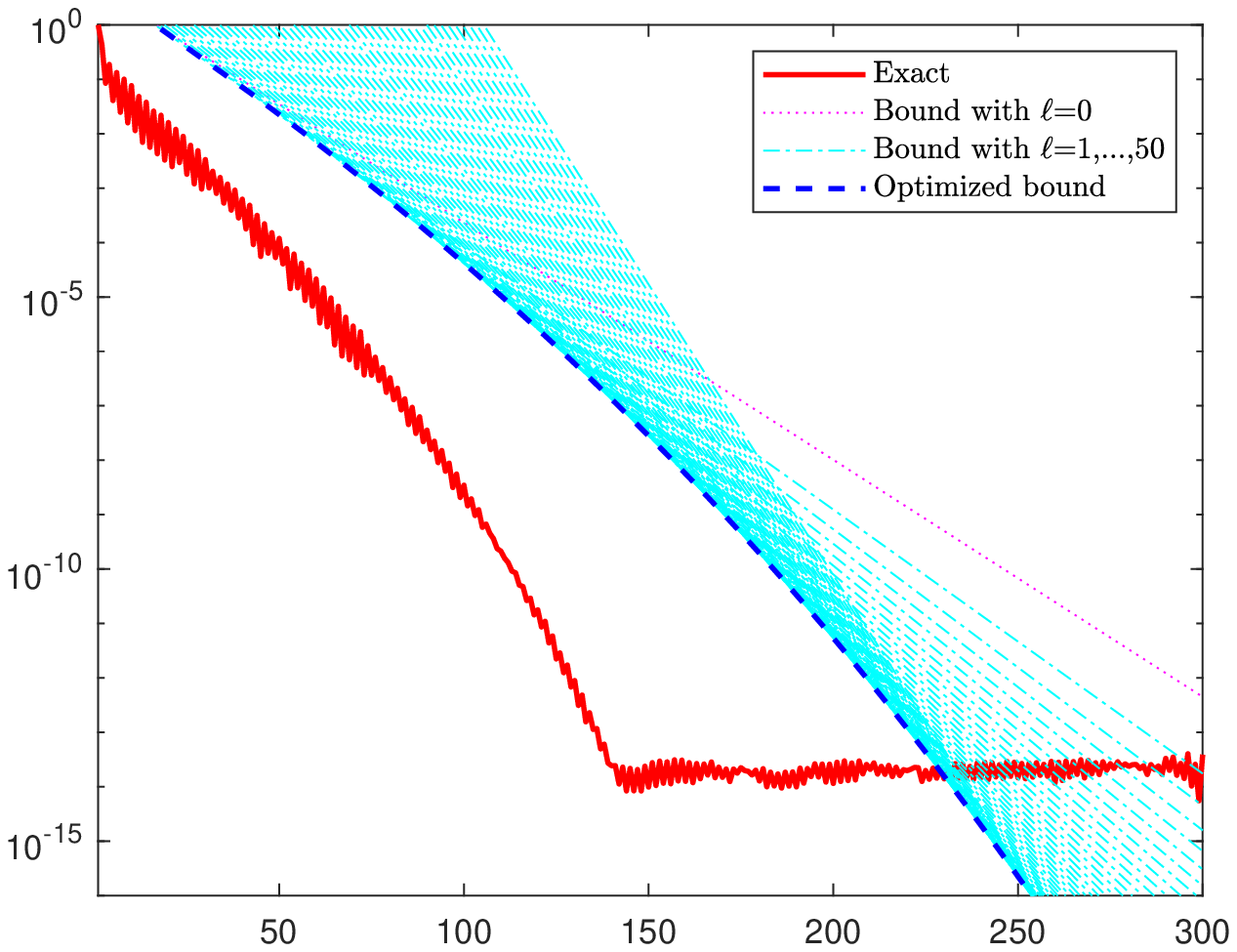}
\end{minipage}
	\caption{Left: Plot of the spectrum of $H$. In this case no symmetry is present but we still have the cluster near the spectral gap.  Right: Exact decay of the projector compared with the bounds (\ref{eqn:projboundSL}) for $\ell=0,1,\dots,50$. 
		\label{fig:superexponential}}
\end{figure}


\section{Conclusions}
We have developed new computable bounds for the entries of spectral projectors which improve and refine the existing ones. The first one has the advantage to be a single bound and not a parametrized family and describes well the decay rate. The second one is asymptotically optimal in the sense of polynomial approximation, although it is not as easy to compute as the first one.

We have also shown that, like for the matrix inverse, the decay properties of the projector are connected to the full spectral information. As a result, we are able to predict superexponential  decay behaviour in the presence of isolated eigenvalues at the extremes of the spectrum.


\begin{thebibliography}{10}
	
	\bibitem{Baskakov}
	A.~G. Baskakov.
	\newblock Estimates for the elements of inverse matrices, and the spectral
	analysis of linear operators.
	\newblock {\em Izv. Ross. Akad. Nauk Ser. Mat.}, 61(6):3--26, 1997.
	
	\bibitem{BenziCime}
	M.~Benzi.
	\newblock Localization in matrix computations: theory and applications.
	\newblock In {\em Exploiting Hidden Structure in Matrix Computations:
		Algorithms and Applications}, volume 2173 of {\em Lecture Notes in
		Mathematics}, pages 211--317. Springer, Cham, 2016.
	
	\bibitem{BenziFOV}
	M.~Benzi.
	\newblock Some uses of the field of values in numerical analysis.
	\newblock {\em Boll. Unione Mat. Ital.}, 14(1):159--177, 2021.
	
	\bibitem{Igor}
	M.~Benzi, D.~Bertaccini, F.~Durastante, and I.~Simunec.
	\newblock Non-local network dynamics via fractional graph {L}aplacians.
	\newblock {\em J. Complex Netw.}, 8(3):cnaa017, 29, 2020.
	
	\bibitem{BBR}
	M.~Benzi, P.~Boito, and N.~Razouk.
	\newblock Decay properties of spectral projectors with applications to
	electronic structure.
	\newblock {\em SIAM Rev.}, 55(1):3--64, 2013.
	
	\bibitem{BenziGolub}
	M.~Benzi and G.~H. Golub.
	\newblock Bounds for the entries of matrix functions with applications to
	preconditioning.
	\newblock {\em BIT}, 39(3):417--438, 1999.
	
	\bibitem{BenziRazouk}
	M.~Benzi and N.~Razouk.
	\newblock Decay bounds and {$O(n)$} algorithms for approximating functions of
	sparse matrices.
	\newblock {\em Electron. Trans. Numer. Anal.}, 28:16--39, 2007/08.
	
	\bibitem{BenziSimoncini}
	M.~Benzi and V.~Simoncini.
	\newblock Decay bounds for functions of {H}ermitian matrices with banded or
	{K}ronecker structure.
	\newblock {\em SIAM J. Matrix Anal. Appl.}, 36(3):1263--1282, 2015.
	
	\bibitem{Miyazaki}
	D.~Bowler and T.~Miyazaki.
	\newblock ${O}({N})$ methods in electronic structure calculations.
	\newblock {\em Rep. Prog. Phys. Physical Society (Great Britain)}, 75:036503,
	03 2012.
	
	\bibitem{ChuiHasson}
	C.~K. Chui and M.~Hasson.
	\newblock Degree of uniform approximation on disjoint intervals.
	\newblock {\em Pacific J. Math.}, 105(2):291--297, 1983.
	
	\bibitem{Demko}
	S.~Demko, W.~F. Moss, and P.~W. Smith.
	\newblock Decay rates for inverses of band matrices.
	\newblock {\em Math. Comp.}, 43(168):491--499, 1984.
	
	\bibitem{Diestel}
	R.~Diestel.
	\newblock {\em Graph Theory}, volume 173 of {\em Graduate Texts in
		Mathematics}.
	\newblock Springer, Berlin, fifth edition, 2018.
	\newblock Paperback edition of [ MR3644391].
	
	\bibitem{Eremenko2007}
	A.~Eremenko and P.~Yuditskii.
	\newblock Uniform approximation of {${\rm sgn}\,x$} by polynomials and entire
	functions.
	\newblock {\em J. Anal. Math.}, 101:313--324, 2007.
	
	\bibitem{Eremenko2011}
	A.~Eremenko and P.~Yuditskii.
	\newblock Polynomials of the best uniform approximation to {${\rm sgn}(x)$} on
	two intervals.
	\newblock {\em J. Anal. Math.}, 114:285--315, 2011.
	
	\bibitem{Frommer2017}
	A.~Frommer, C.~Schimmel, and M.~Schweitzer.
	\newblock Bounds for the decay of the entries in inverses and
	{C}auchy-{S}tieltjes functions of certain sparse, normal matrices.
	\newblock {\em Numer. Linear Algebra Appl.}, 25(4):e2131, 17, 2018.
	
	\bibitem{Frommer2018}
	A.~Frommer, C.~Schimmel, and M.~Schweitzer.
	\newblock Non-{T}oeplitz decay bounds for inverses of {H}ermitian positive
	definite tridiagonal matrices.
	\newblock {\em Electron. Trans. Numer. Anal.}, 48:362--372, 2018.
	
	\bibitem{Frommer2021}
	A.~Frommer, C.~Schimmel, and M.~Schweitzer.
	\newblock Analysis of {P}robing {T}echniques for {S}parse {A}pproximation and
	{T}race {E}stimation of {D}ecaying {M}atrix {F}unctions.
	\newblock {\em SIAM J. Matrix Anal. Appl.}, 42(3):1290--1318, 2021.
	
	\bibitem{Fuchs}
	W.~H.~J. Fuchs.
	\newblock On the degree of {C}hebyshev approximation on sets with several
	components.
	\newblock {\em Izv. Akad. Nauk Armyan. SSR Ser. Mat.}, 13(5-6):396--404, 541,
	1978.
	
	\bibitem{Golub}
	G.~H. Golub and C.~F. Van~Loan.
	\newblock {\em Matrix Computations}.
	\newblock Johns Hopkins Studies in the Mathematical Sciences. Johns Hopkins
	University Press, Baltimore, MD, fourth edition, 2013.
	
	\bibitem{Hasson}
	M.~Hasson.
	\newblock The degree of approximation by polynomials on some disjoint intervals
	in the complex plane.
	\newblock {\em J. Approx. Theory}, 144(1):119--132, 2007.
	
	\bibitem{Higham}
	N.~J. Higham.
	\newblock {\em Functions of Matrices}.
	\newblock Society for Industrial and Applied Mathematics (SIAM), Philadelphia,
	PA, 2008.
	\newblock Theory and computation.
	
	\bibitem{Iserles}
	A.~Iserles.
	\newblock How large is the exponential of a banded matrix?
	\newblock {\em New Zealand J. Math.}, 29(2):177--192, 2000.
	\newblock Dedicated to John Butcher.
	
	\bibitem{Kohn}
	W.~Kohn.
	\newblock Density functional and density matrix method scaling linearly with
	the number of atoms.
	\newblock {\em Phys. Rev. Lett.}, 76:3168--3171, Apr 1996.
	
	\bibitem{Strakos}
	J.~Liesen and Z.~Strako\v{s}.
	\newblock {\em Krylov Subspace Methods}.
	\newblock Numerical Mathematics and Scientific Computation. Oxford University
	Press, Oxford, 2013.
	\newblock Principles and analysis.
	
	\bibitem{Meinardus}
	G.~Meinardus.
	\newblock {\em Approximation of Functions: Theory and Numerical Methods}.
	\newblock Expanded translation of the German edition. Translated by Larry L.
	Schumaker. Springer Tracts in Natural Philosophy, Vol. 13. Springer-Verlag
	New York, Inc., New York, 1967.
	
	\bibitem{Niklasson2011}
	A.~M.~N. Niklasson.
	\newblock Density matrix methods in linear scaling electronic structure theory.
	\newblock In R.~Zalesny, M.~G. Papadopoulos, P.~G. Mezey, and J.~Leszczynski,
	editors, {\em Linear-Scaling Techniques in Computational Chemistry and
		Physics: Methods and Applications}, pages 439--473. Springer Netherlands,
	Dordrecht, 2011.
	
	\bibitem{PozzaSimoncini}
	S.~Pozza and V.~Simoncini.
	\newblock Inexact {A}rnoldi residual estimates and decay properties for
	functions of non-{H}ermitian matrices.
	\newblock {\em BIT}, 59(4):969--986, 2019.
	
	\bibitem{Riascos}
	A.~P. Riascos and J.~L. Mateos.
	\newblock Fractional dynamics on networks: Emergence of anomalous diffusion and
	{L}\'evy flights.
	\newblock {\em Phys. Rev. E}, 90:032809, Sep 2014.
	
\end{thebibliography}
\end{document}